\documentclass[11pt, a4paper]{amsart}              % Book class in 11 points
\author{Huda Alzaki and Jae-Cheon Joo} 
\address{Department of Mathematics, King Fahd University of Petroleum and Minerals, 31261 Dhahran, Saudi Arabia}
\email{g201901330@kfupm.edu.sa; jcjoo@kfupm.edu.sa}             %   for the title page.
\date{\today}                           %   Use current date. 

\usepackage{amsmath, amsthm, amsfonts, amssymb}

\newcommand{\re}{\mathrm{Re}\,}
\newcommand{\IM}{\mathrm{Im}\,}
\newcommand{\Aut}{\mathrm{Aut}}

\newcommand{\C}{\mathbb{C}}
\newcommand{\R}{\mathbb{R}}

\newcommand{\beq}{\begin{eqnarray*}}
\newcommand{\eeq}{\end{eqnarray*}}
\newcommand{\beg}{\begin{equation}}
\newcommand{\eeg}{\end{equation}}

\newcommand{\Om}{\Omega}
\newcommand{\di}{\partial}

\newcommand{\la}{\lambda}

\newcommand{\al}{\alpha}

\newcommand{\ov}{\overline}
\newcommand{\ii}{\sqrt{-1}}
\newcommand{\vp}{\varphi}

\newcommand{\wt}{\widetilde}
\newcommand{\mc}{\mathcal}
\newcommand{\rw}{\rightarrow}

\newtheorem{thm}{Theorem}[section]

\newtheorem{lem}[thm]{Lemma}
\newtheorem{prop}[thm]{Proposition}
\theoremstyle{definition}
\newtheorem{definition}[thm]{Definition}
\theoremstyle{remark}
\newtheorem{remark}[thm]{Remark}
\newtheorem{example}[thm]{Example}

\numberwithin{equation}{section}

\title[Indicatrix and Maximal Circularity]{Indicatrix of Invariant metrics, Maximal circularity and Scaling of domains}

\subjclass[2010]{32H02, 32M05}
\keywords{Invariant metrics, Indicatrix, Scaling, Circularity}

\begin{document}

\begin{abstract}
We investigate how to use the indicatrix of an invariant metric to rescale a sequence of biholomorphic maps and to ensure the convergence of the rescaled sequence. We also use the indicatrix of the Kobayashi-Royden metric to define the maximal circularity function which indicates how the domain is close to a circular domain.
\end{abstract}

\maketitle

\section{Introduction}
This paper consists of several observations on the celebrated work \cite{Fr} by S. Frankel which states that a convex hyperbolic domain in complex Euclidean space admitting a discrete and cocompact subgroup of automorphisms which acts freely is biholomorphic to a bounded symmetric domain. As a key ingredient, Frankel proved that if $\Om$ is a convex hyperbolic domain, then for any sequence $\{\vp_j\}$ of holomorphic automorphisms and any  $p\in \Om$, the {\em scaled} sequence
\beg\label{e;Fr1}
\tau_j = [d\vp_j (p)]^{-1} \circ (\vp_j(\cdot) - \vp_j(p)),
\eeg
has a convergent subsequence. In this paper, we investigate the sequence \eqref{e;Fr1} about the construction, the convergence and the symmetry of limits in a geometric point of view on invariant metrics.  

Roughly speaking, by the method of scaling we mean a way of renomalizing maps in a sequence by compositions with affine linear maps to make the new sequence subseuentially convergent. The linear map which is composed for the scalng is usually called the {\em stretching factor}. Although the idea of scaling is used in a wide range of mathematics, two specific constructions are usually cited when we deal with sequences of holomorphic automorphisms of a domain: Frankel's sequence \eqref{e;Fr1} and Pinchuk's one (\cite{Pin}). In contrast with \eqref{e;Fr1}, Pinchuk's construction of the stretching factor is more complicated and not completely known in general cases of domain. On the other hand, Pinchuk's construction has a merit that the stretching factor is completely determined by the geometry of domain, while the stretching factor of \eqref{e;Fr1} is $(d\vp_j(p))^{-1}$ and it depends on the map. The independence of stretching factor from the map in Pinchuk's constuction turned out to be quite useful in dealing with invariants of domain, and it has been asked if or when two constuctions are equivalent. There are some partial results about the question (\cite{SJoo, KK}), however it is not answered yet in general cases. Instead, we construct a new strecthing factor  which is fully detemined by the indicatrix of an invariant metric and prove that the scaled sequence with the new stretching factor is indeed equivalent with \eqref{e;Fr1} as the first result (Theorm \ref{t;main1}). 

The second observation is on a condition for the convergence of \eqref{e;Fr1}. As mentioned above, Frankel proved that the sequence is subsequentially convergent if $\Om$ is convex and hyperbolic by obtaining local uniform estimates of derivatives for automorphisms up to the second order. In view of Theorem \ref{t;main1}, it is also natural to describe a geometric condition on invariant metrics for the convergence. The condition is given by a relation between the indicatrix of an invariant metric and the ball of the induced distance. We call the condition the {\em uniform domination} of the domain by the given metric (Definition \ref{d;dom}). After observing the convergence of \eqref{e;Fr1} under the uniform domination condition, we will prove that any convex hyperbolic domain is uniformly dominated by the Kobayashi-Royden metric in Theorem \ref{t;ud} in order to recover the case of \cite{Fr}. 

The last result is about a symmetry of the image of a subsequential limit of \eqref{e;Fr1}. In \cite{Fr}, the author chooses a sequence of automorphisms which drags an interior point to a {\em regular} boundary point, and  proves the image of the limit (and hence the given domain as well) admits a $1$-parameter family of automorphisms. Let us recall that if $q$ is a regular boundary point of a bounded symmetric domain which is not biholomorphic to the ball, then there is a nontrivial analytic variety in the boundary which passes through $q$. On the other hand, a boundary point $q$ of a domain is called a {\em minimal} boundary point if the boundary of domain does not contain any nontrivial analytic variety passing through $q$. As seen in some examples of domain of piecewise smooth boundary  (\cite{Kim2, KKS, KP, Zim}), the image of the limit of the scaled sequence is usually more canonical in case the given sequence of automorphisms drags an interior point to a minimal boundary point rather than to a non-minimal boundary point. With this observation, it is natural to ask {\em what kind of symmetry the limit of \eqref{e;Fr1} has, in case $\Om$ is bounded and convex, and $\vp_j(p)$ tends to a minimal boundary point.} Another question is {\em where the circularity of the domain comes from.} Recall that the circularity is a signature property of bounded symmetric domains in the set of bounded homogeneous domains. We believe that these two questions are indeed same and can be combined to the following one.

{\em Question.} Let $\Om$ be a bounded convex domain in $\C^n$. Suppose that there exists a sequence $\{\vp_j\}$ of holomorphic automorphisms and $p\in \Om$ such that $\vp_j(p)$ tends to a minimal boundary point of $\Om$. Then is $\Om$  biholomorphic to a completely circular domain?

In order to study this problem in more geometric point of view, we define in Section 3 an invariant function $c_\Om$ on $\Om$ which we call the {\em maximal circularity}. We will see that $0< c_\Om(x)\leq 1$ for any $x\in \Om$ and $c_\Om(x) =1$ if and only if $\Om$ is biholomorphic to the indicatrix of the Kobayashi-Royden metric of $\Om$ at $x$. With this, we can ask a more general question.

{\em Question.} Let $\Om$ be a bounded convex domain in $\C^n$. Does $c_\Om (x)$ tend to $1$ as $x$ tends to a minimal boundary point?

We do not know yet whether or not the answers to the questions are all yes. We will only prove in Section 3 that the answer to the last question is affirmative in case $\Om$ is either strongly pseudoconvex (not necessarily convex) or a convex analytic polyhedral domain.

\bigskip

{\em Acknowledgement.} We would like to thank Professor Kang-Tae Kim for many inspiring discussions about the scaling theory.

\section{Construction and Convergence of a scaling sequence}

 Let $\Om$ be a domain in $\C^n$. By a {\em complex Finsler metric} or simply a {\em metric} on $\Om$, we mean an upper semicontinuous (usc) function $F$ on $T\Om = \Om\times \C^n$ such that 
\begin{itemize}
\item[(\romannumeral1)] $F(x; v) \geq 0$ for any $x\in \Om$ and $v\in\C^n$ and $F(x; v) = 0$ only if $v=0$, and 

\item[(\romannumeral2)] $F(x; cv) = |c| F(x; v)$ for any $(x,v) \in \Om\times \C^n$ and $c\in \C$.
\end{itemize}
The {\em indicatrix} of $F$ at $x\in \Om$ is the set 
$$I_{\Om, F}(x) := \{v\in \C^n : F(x;v) <1\}.$$ We say a metric on $\Om$ is {\em convex} if the indicatrix of $F$ is a convex domain in $\C^n$ at any point of $\Om$. 

\begin{example}
Obviously, any Hermitian metric on $\Om$ is a convex metric. In particular, the Bergman metric or the K\"ahler-Einstein metric are convex metrics which are invariant under biholomorphic maps. Another example of invariant metric on a domain is the Kobayashi-Royden metric which is defined by 
$$K_\Om (x; v) = \inf \{\la >0 : \exists f\in \mathcal{O}(\Delta, \Om),\,\, f(0) =x,\,\, \la f'(0) = v\},$$
where $\mathcal O (\Delta, \Om)$ represents the space of holomorphic mappings from the unit disc $\Delta$ into $\Om$. Indeed, $K_\Om$ is not necessarily a metric, since $K_\Om(x; v)$ may be $0$ for a nonzero vector $v$, however, it turns out that it is a metric for instance if $\Om$ is bounded. We call a domain (or more generally a complex manifold) a {\em hyperbolic domain} if $K_\Om$ is a metric. It turns out that $K_\Om$ is convex if $\Om$ is convex however it is not true for general hyperbolic domains. To make it convex, one can consider a new invariant metric $\widehat K_\Om$ whose indicatrix is the convex hull of the indicatrix of $K_\Om$ at each point. The metric is called the {\em Kobayashi-Buseman metric.} See for instance \cite{JP, Ko} for more details about invariant metrics.
\end{example}

Let $F$ be a metric on a domain $\Om$ and $x\in\Om$. By the upper semicontnuity, we can find $u_1 =u_1(x)\in \C^n$ with $|u_1|=1$ such that 
$$\mu_1(x) := F(x;u_1) = \max\{F(x; v): v\in\C^n,\,\,|v|=1\}.$$ 
Inductively, if we have chosen $u_1,\ldots, u_\al$ and $\mu_\beta = F (x; u_\beta)$ for $1\leq\beta\leq\al\leq n-1$, then a unit vector $u_{\al+1}$ and $\mu_{\al+1}$ are defined by 
$$\mu_{\al+1}(x) := F(x;u_{\al+1}) = \max\{F(x; v): v\in V_\al,\,\,|v|=1\},$$
where $V_\al$ is the orthogonal complement of $\{u_1\ldots,u_\al\}$ in $\C^n$.
Let $L_x$ be the linear transform on $\C^n$ defined by 
\beg\label{e;str}
L_x(u_\al) = \mu_\al(x) e_\al, \quad \al=1,\ldots,n,
\eeg
where $\{e_1,\ldots,e_n\}$ is the standard Euclidean basis for $\C^n$. We call $L_x$ a {\em stretching factor} of $\Om$ at $x$ induced by the metric $F$.

For a sequence $\vp_j\in \Aut(\Om)$, the group of holomorphic automorphisms of $\Om$, S. Frankel (\cite{Fr}) defined a sequence
\beg\label{e;Fr}
\tau_j = [d\vp_j (p)]^{-1} \circ (\vp_j(\cdot) - p_j),
\eeg
where $p_j = \vp_j(p)$ and $d\vp_j (p)$ represents the complex Jacobian matrix of $\vp_j$ at $p$. Frankel proved that the sequence $\{\tau_j\}$ has a convergent subsequence if $\Om$ is hyperbolic and convex. We prove that the linear transform $[d\vp_j (p)]^{-1}$ can be replaced by a stretching factor induced by an invariant metric.

\begin{thm}\label{t;main1}
Let $\Om$ be a domain in $\C^n$, $p\in \Om$ and $\{\vp_j\}$ be a sequence of holomorphic automorphisms. Then the sequence \eqref{e;Fr} converges subsequentially if and only if 
\beg\label{e;scal}
\sigma_j:= L_j \circ (\vp_j (\cdot) -p_j)
\eeg converges subsequentially, where $L_j = L_{p_j}$ is a stretching factor induced by any convex invariant metric.
\end{thm}

\begin{proof}
Let $F$ be a convex invariant metric on $\Om$. Suppose that $u_\al^j$ and $\mu_\al^j$ were chosen to construct $L_j$ by \eqref{e;str} at $p_j = \vp_j(p)$ and let
$$\tilde u^j_\al = d\vp_j(p)^{-1}(u^j_\al),\quad \al= 1,\ldots,n.$$
Since $\vp_j$ preserves $F$, we see that $F(p; \tilde u^j_\al) = \mu^j_\al$. Therefore,
the vectors $\tilde u^j_\al/\mu^j_\al \in \di I_{\Om,F} (p),$ where
$$I_{\Om,F}(q) = \{v\in\C^n : F(q; v)< 1\},$$ the indicatrix of $F$ at $q\in\Om$. Notice that the indicatrix is always convex for any $q\in \Om$. Let $A_j$ be the linear map on $\C^n$ defined by 
$$A_j \left(\frac{\tilde u^j_\al}{\mu^j_\al}\right) = e_\al,\quad \al=1,\ldots,n.$$
Then 
$$L_j = A_j\circ d\vp_j(p)^{-1}$$
by definition. Therefore, we only need to show that 
\beg\label{e;est}
C_1 |v| \leq |A_j v| \leq C_2 |v|,\quad \forall v\in\C^n
\eeg for some constants $C_1>0$ and $C_2>0$ which are independent of $j$, in order to show the equivalence of the subsequential convergence of $\{\sigma_j\}$ and $\{\tau_j\}$. 

 The lower estimate is equivalent to the upper bound of $A_j^{-1}$. This is obvious since
$$|A_j^{-1}v| = \left|A_j^{-1}\left(\sum_\al a_\al e_\al\right)\right|= \left|\sum_\al a_\al \left(\frac{\tilde u^j_\al}{\mu^j_\al}\right) \right|\leq nM |v|$$
for any $v = \sum_\al a_\al e_\al\in\C^n$, where $M = \sup\{|w| : w\in \di I_{\Om,F}(p)\}$. 

This lower estimate and Cramer's formula of inverse transform also implies that the upper estimate in \eqref{e;est} follows from an estimate of the determinant of $A_j$, that is, the upper estimate is obtained if we prove 
\beg\label{e;det}
|\det A_j| \leq C
\eeg for some constant $C>0$ independent of $j$. The definition of $L_j$ implies 
\beg\label{e;Lj}
|\det L_j| = \prod_\al\mu^j_\al.
\eeg Let $\det_\R$ be the determinant of real linear operator on $\R^{2n}=\C^n$. Then 
\beg\label{e;dvp}
{\det}_\R d\vp_j(p) = |\det d\vp_j(p)|^2 = \frac{\mathrm{Vol}(I_{\Om,F}(p_j))}{\mathrm{Vol}(I_{\Om,F}(p))} = C\,\mathrm{Vol}(I_{\Om,F}(p_j))
\eeg
where $\mathrm{Vol}(E)$ represents the Euclidean volume of a given region $E$ and $C^{-1}=\mathrm{Vol}(I_{\Om,F}(p))$, since $d\vp_j(p)$ maps $I_\Om(p)$ onto $I_\Om(p_j)$. Identifying the vector $\sum_\al a_\al u^j_\al$ with the $n$-tuple $(a_1,\ldots,a_n)$, we regard $I_{\Om,F}(p_j)$ as a balanced convex set in $\C^n$. Regarding $\C^n = \R^n + \sqrt{-1}\R^n$, let
$$R_j =I_{\Om,F}(p_j) \cap \R^n =  \{(a_1,\ldots,a_n)\in I_{\Om,F}(p_j) : a_\al\mbox{'s are real} \}.$$
Since $I_{\Om,F}(p_j)$ is completely circular, it is a circular region generated by $R_j$ and $R_j$ is symmetric with respect to the origin. Notice that  $R_j$ is also a convex region in $\R^n$ since $I_{\Om,F}(p_j)$ is convex. In order to prove the estimate \eqref{e;det}, we need the following lemma.

\begin{lem}\label{l;box} Let $R$ be a bounded convex domain in $\R^n$ which is symmetric with respect to the origin. Let 
$$r_1 = \min\{|x| : x\in \di R\}.$$
Suppose that $(r_1,0,\ldots, 0)\in \di R$ (this means the minimum distance to the boundary is attained at the $x_1$-intercept of $\di R$). Let 
$$\wt R = \{(x_2,\ldots, x_n) \in\R^{n-1} : (0, x_2,\ldots, x_n) \in R\}.$$ Suppose 
$$\wt R \subset \wt B :=[-r_2, r_2]\times\cdots\times[-r_n, r_n]$$ for some $r_2,\ldots, r_n >0$. Then 
$$R \subset B :=[-r_1, r_1]\times 2\wt B = [-r_1, r_1]\times[-2r_2, 2r_2]\times\cdots\times[-2r_n, 2r_n].$$
\end{lem}
\begin{proof}[Proof of Lemma \ref{l;box}]
We notice first that the supporting hyperplanes at $(\pm r_1,0,\ldots,0)$ are 
$$x_1 = \pm r_1,$$
since $r_1$ is the minimum distance from the origin to $\di R$. Therefore
$$x=(x_1,\ldots,x_n)\in R \Longrightarrow |x_1| \leq r_1.$$
We also observe that we may assume $\wt R = \wt B$. Otherwise, we only need to replace $R$ by the convex hull of $\wt B \cup R$. Since $\di \wt R = \di \wt B$ contains boxes $\{\pm r_2\}\times[-r_3, r_3]\times\cdots\times[-r_n, r_n]$, the supporting hyperplanes of $R$ at $(0, \pm r_2, 0,\ldots, 0)\in \di R$ also contains those boxes. This means if 
$$x_2 = \al_1 x_1 + \al_3 x_3 + \cdots + \al_n x_n \pm r_2$$ 
are the equation of the supporting hyperplanes, then $\al_3=\cdots=\al_n=0.$ Then from the condition that the distance from the origin to the supporting hyperplanes are greater than equal to $r_1$, we have 
$$\frac{r_2}{\sqrt{\al_1^2 + 1}} \geq r_1.$$ 
Thereore, 
$$|\al_1| \leq \sqrt{\left(\frac{r_2}{r_1}\right)^2 -1}.$$
Since $x=(x_1,\ldots,x_n)$ lies between two support hyperplanes $x_2 = \al_1 x_1 \pm r_2$ and $|x_1|\leq r_1$ if $x\in R$, we conclude that 
$$|x_2|\leq \sqrt{r_2^2 - r_1^2} + r_2 \leq 2r_2.$$ Similarly, we can also show that $|x_\al| \leq 2r_\al$ for any $\al=2,\ldots,n$ whenever $x=(x_1,\ldots,x_n) \in R$. This yields the conclusion.
\end{proof}
We resume the proof of Theorem \ref{t;main1}. Let 
$$R_{j,1} = \{(x_{n-1}, x_n) \in \R^2 : (0,\ldots,0, x_{n-1},x_n)\in R_j\}.$$ 
Applying Lemma \ref{l;box}, we see that 
$$R_{j,1} \subset B_{j,1} := [-1/\mu^j_{n-1}, 1/\mu^j_{n-1}]\times [-2/\mu^j_n, 2/\mu^j_n],$$ since
$\wt R_{j,1} = [-1/\mu^j_n, 1/\mu^j_n].$ Applying the lemma to 
$$R_{j,2} = \{(x_{n-2}, x_{n-1}, x_n) \in \R^3 : (0,\ldots,0,x_{n-2}, x_{n-1},x_n)\in R_j\},$$ we see that 
$$R_{j,1} \subset B_{j,2}: = [-1/\mu^j_{n-2}, 1/\mu^j_{n-2}]\times [-2/\mu^j_{n-1}, 2/\mu^j_{n-1}]\times[-2^2/\mu^j_n, 2^2/\mu^j_n],$$ since $\wt R_{j,2} = R_{j,1}\subset B_{j,1}.$ Repeating this procedure, we conclude that 
$$R_j \subset B_j:= \prod_{\al=1}^n [-2^{\al-1}/\mu^j_\al, 2^{\al-1}/\mu^j_\al].$$ Therefore, the indicatrix $I_{\Om,F}(p_j)$ is a subset of the circular region in $\C^n$ generated by $B_j$ which is the polydisc 
$$P_j := \{z=(z_1,\ldots,z_n)\in \C^n: |z_\al|\leq 2^{\al-1}/\mu^j_\al, \,\,\al=1,\ldots, n\}.$$ This implies that 
$$\mathrm{Vol}(I_{\Om,F}(p_j))\leq \mathrm{Vol}(P_j)\leq C \prod_\al \left(\frac{1}{\mu^j_\al} \right)^2$$ for some constant $C>0$ independent of $j$. Then \eqref{e;Lj} and \eqref{e;dvp} yields the desired estimate \eqref{e;det}.
\end{proof}

We will now investigate a condition of metric under that the Frankel sequence \eqref{e;Fr} has a convergent subsequence. Let $F$ be a metric on a domain $\Om$. We denote by $d_{\Om, F}$ the induced distance function, that is, 
$$d_{\Om, F} (x,y) = \inf\left\{\int_0^1 F(c(t); c'(t))\,dt : c\in C(x,y)\right\}$$
where $C(x,y)$ denotes the space of all piecewise differentiable (real) curve in $\Om$ joining $x$ and $y$. For $r>0$, we denote the $d_{\Om,F}$-ball of radius $r$ by $B_{\Om, F}(x;r)$, that is,
$$B_{\Om, F} (x; r)  = \{y\in \Om : d_{\Om, F}(x,y) <r\}.$$

\begin{definition}\label{d;dom}
Let $\Om$ be a domain in $\C^n$. We say that $\Om$ is {\em uniformly dominated } by a metric $F$ on $\Om$ if 
\begin{itemize}
\item $F$ is complete (i.e. $d_{\Om, F}$ is Cauchy complete),

\item For each $r>0$, there is $\la(r) >0$ such that 
\beg\label{e;dom}
B_{\Om, F}(x;r) -x \subset \la(r) I_{\Om, F}(x)
\eeg
for any $x\in \Om$, where $B_{\Om, F}(x; r) -x = \{y-x : y\in B_{\Om, F}(x; r)\}$, the Euclidean translation.
\end{itemize}
\end{definition}

\begin{prop}\label{p;normal}
Let $\Om$ be a domain and $p\in\Om$ fixed. Assume that $\Om$ is uniformly dominated by an invariant metric $F$. Then 
$$\mathcal S = \{\tau_\vp:\Om\rw \C^n : \vp\in \mathrm{Aut}(\Om)\}$$ is a normal family, where
$\tau_\vp = [d\vp (p)]^{-1} \circ (\vp(\cdot) - \vp(p))$,
\end{prop}

\begin{proof}
For any $r>0$ and $\vp\in\Aut (\Om)$, 
\beq \tau_\vp(B_{\Om, F}(p; r)) &=& [d\vp (p)]^{-1}(B_{\Om, F}(\vp(p); r)-\vp(p))\\
&\subset&\la(r)[d\vp (p)]^{-1}( I_{\Om, F}(\vp (p)) )\\
&=&\la(r) I_{\Om, F}(p)
\eeq
by \eqref{e;dom}, since $F$ is invariant. Therefore, the family $\mathcal S$ is uniformly bounded on $B_{\Om, F}(p; r)$. Since $F$ is complete, this implies that $\mathcal S$ is uniformly bounded on each compact subset of $\Om$. The Montel theorem yields the conclusion.
\end{proof}

\begin{example}\label{ex;hp}
Let $\mc H = \{z\in \C : \IM z >0\}$ the upper half plane. Let $F$ be the Poincar\'e metric on $\mc H$ which coincides with the Kobayashi-Royden metric $K_{\mc H}$. The explicit formulas of the metric and the distance are as follows.
\beg\label{e;Poi}
F(z;v) = \frac{|v|}{2\IM z},\quad d_{\mc H, F}(z_1, z_2) = 2 \tanh^{-1}\left|\frac{z_2-z_1}{z_2-\bar z_1}\right|.
\eeg
We will show \eqref{e;dom} for a suitable $\la$ at every $x= a+\ii b.$ By the translation invariance of $F$, we assume without loss of generality $x = \ii b$ for $b>0$. Therefore, the indicatrix $I_{\mc H,F}(x)$ is the disc $\Delta(2b) = \{\zeta\in\C: |\zeta| <2b\}$ and the distance ball $B_{\mc H, F}(x;r)$ is the disc centered at $\ii b(r)$ of radius $\sqrt{b(r)^2 -b^2}$ where 
$$b(r) = \left(\frac{1+\tanh^2(r/2)}{1-\tanh^2 (r/2)}\right) b.$$ Therefore, if we define $\la_{\mc H}(r)$ by 
\beg\label{e;la}
\la_{\mc H}(r) = (2b)^{-1} \left(b(r) -b + \sqrt{b(r)^2 -b^2}\right) = \frac{\tanh(r/2)}{1-\tanh(r/2)},
\eeg
then, we see that $B_{\mc H,F}(x;r) - x\subset \la_{\mc H}(r) I_{\mc H, F}(x)$ for any $x=\ii b\in \mc H$. Since $b$ is arbitrary, this implies that $\mc H$ is uniformly dominated by the Poincar\'e metric.
\end{example}

\begin{remark}\label{r;dom}
It may be necessary to point out that the uniform domination condition of a metric is not holomorphic invariant. More precisely, even if $\Om$ is uniformly dominated by a metric $F$ and $\vp:\widehat \Om \rw \Om$ is a biholomorphic map, $\widehat\Om$ is not necessarily uniformly dominated by $\vp^*F$. The reason is that the definition of the uniform domination depends on a Euclidean translation of sets, which is not preserved by biholomorphic maps in general. Therefore, the uniform domination condition is just invariant under affine transforms. This is not an unexpected result, since the convergence of Frankel's scaling sequence is not a holomorphic invariant property, either.  It is also an affine invariant one.
\end{remark}

It would be an interesting question to ask about conditions of invariant metrics and domains which imply the uniform domination. The following theorem is an answer to this question in case of the Kobayshi-Royden metric. Together with Proposition \ref{p;normal}, it provides a geometric proof of the convergence of Frankel's scaling sequence on convex domains.

In order to simplify symbols, we denote the indicatrix and the distance ball by $I_\Om(x)$ and $B_\Om(x;r)$, respectively, if the metric is the Kobayashi-Royden metric. 

\begin{thm}\label{t;ud}
Let $\Om$ be a hyperbolic convex domain in $\C^n$. Then $\Om$ is uniformly dominated by the Kobayashi-Royden metric $K_\Om$. Indeed, 
$$B_\Om (x;r)-x \subset2 \la_{\mc H}(r) I_\Om (x)$$
for each $r>0$ and $x\in \Om$, where $\la_{\mc H}$ is the function defined as \eqref{e;la}
\end{thm}

\begin{proof}
Let $x\in \Om$ be a point and $v\in\C^n$ be arbitrary nonzero vector. We denote by $l_v$ the complex line defined by 
$$l_v = \{\zeta v : \zeta \in \C\}$$
and $l_v + x = \{x+ \zeta v : \zeta \in \C\}$ by $l_{x,v}$. Let $\Om_{x,v} = \Om \cap l_{x,v}$. Since $\Om$ is convex and hyperbolic, $\Om_{x,v}$ is also a convex hyperbolic planar domain. Let $p = p(x,v)$ be a point on $\di\Om_{x,v}$ such that 
$$d:=|x -p| = \mathrm{dist}(x, \di\Om_{x,v})$$ where $\mathrm{dist}$ represents the Euclidean distance. Then since $\Om_{x,v}$ contains the disc centered at $x$ of radius $d$, 
\beg\label{e;up}
K_\Om(x; v) \leq K_{\Om_{x,v}}(x; v) \leq \frac{|v|}{d} = 2 K_{\mc H}(\ii d; v).
\eeg
By a unitary change which map $(x-p)/|x-p|$ to $e_1$ and the translation $z\rw z-p$, we may first assume that 
\begin{itemize}
\item $p$ is the origin,

\item $l_{x,v}$ is the $z_1$-plane,

\item $\Om_{x,v}$ lies on the upper half plane in $z_1$-plane and the real axis is the supporting line of $\Om_{x,v}$,

\item $x = \ii d$.
\end{itemize}
Applying a linear transform which leaves $z_1$-plane invariant, we can also assume that 
\begin{itemize}
\item the supporting hyperplane of $\Om$ at $p=O$ is $\Pi=\{\IM z_1 =0\}$ and as a result $\Om \subset {\mc H}\times \C^{n-1}$.
\end{itemize}
It can be achieved by the linear map 
$$e_1\rw e_1,\,\, v_2\rw e_2,\ldots, v_n\rw e_n,$$
where $v_2,\ldots,v_n$ are linearly independent complex vectors which are tangent to $\Pi\cap \ii \Pi$. Recall that the Kobayashi distance decreases under holomorphic map, and if $p_1, p_2 \in X$ and $q_1, q_2\in Y$ for two complex manifolds $X$ and $Y$, then 
$$d_{X\times Y}((p_1, q_1), (p_2, q_2)) = \max\{d_X(p_1, p_2), d_Y(q_1, q_2)\}.$$  
Therefore, if $y\in \Om_{x, v} \subset \mc H\times\{O\}$, then 
\beg\label{e;dup}
d_\Om (x, y) \geq d_{\mc H\times \C^{n-1}} (x, y) = d_{\mc H}(x,y),
\eeg by identiying $\mc H\times \{O\}$ with $\mc H$. Combining \eqref{e;up} and \eqref{e;dup}, we conclude that 
$$(B_\Om(x; r) - x)\cap l_v \subset 2\la_{\mc H}( I_\Om (x) \cap l_v).$$
The conclusion follows since $x\in \Om$ and $v\in \C^n$ are arbitrary.
\end{proof}

\section{Maximal circularity function}
Let $\Om$ be a bounded domain in $\C^n$. For $x\in \Om$, let
$\mc B_\Om(x)$ be the class of bounded holomorphic embeddings $\vp:\Om\rw \C^n$ such that  $\vp(x) =0$. Let 
\begin{eqnarray}\label{e;mcf}
c_\Om (x) &:=& \sup \left\{\frac{r}{R} : \exists \vp \in \mc B_\Om(x), R>0, r>0 \mbox{ such that }\right.\\
& &\hspace{2cm} \left.r I_{\vp(\Om)}(0) \subset \vp (\Om) \subset R I_{\vp (\Om)}(0) \right\},\nonumber
\end{eqnarray}
where $I_\Om (x)$ represents the Kobayashi indicatrix of $\Om$ at $x\in \Om$. Obiously, $0 < c_\Om (x) \leq 1$, and if $\Om$ is a bounded complete circular domain, then $\Om$ is same with its Kobayashi indicatrix at the origin (\cite{Bar}) and hence  $c_\Om (0) =1$. With ths reason, we call $c_\Om$ the {\em maximal circularity function} or simply the {\em maximal circularity} of $\Om$. The first observation is that the converse is also true.

\begin{prop}\label{p;mc1}
Let $\Om$ be a bounded domain in $\C^n$. Suppose that $c_\Om(x) =1$ at a point $x\in \Om$. Then $\Om$ is biholomorphic to a complete circular domain. 
\end{prop}

\begin{proof}
Suppose $\vp_j \in \mc B_\Om(x)$ such that 
\beg\label{e;p1} r_j I_j \subset \vp_j (\Om) \subset R_j I_j\eeg for some $R_j \geq r_j >0$ and $r_j/R_j \rw 1 $ as $j\rw \infty$, where 
$$ I_j := I_{\vp_j(\Om)} (0) = d\vp_j (x) (I_\Om (x)).$$
Let
$$\psi_j = [d\vp_j(x)]^{-1} \circ \vp_j.$$ Then 
$\psi_j \in \mc B_\Om(x)$, and acting $[d\vp_j(x)]^{-1}$ to \eqref{e;p1}, we have
$$r_j I_\Om (x) \subset \psi_j (\Om) \subset R_j I_\Om(x).$$
Notice that $R_j <R$ for some $R>0$. Otherwise, we may assume $R_j\rw \infty$ as $j\rw\infty$. Then $r_j\rw \infty$ as well. Let $\epsilon >0$ be chosen small enough that $B_\epsilon \subset I_\Om (x)$, where $B_\epsilon$ is the Euclidean $\epsilon$-ball centered at the origin. Let $v \in \C^n$ be a nonzero vector. Then since $d\psi_j (x) = Id$ and $\Om$ is hyperbolic, 
$$0< K_\Om (x; v) = K_{\psi_j (\Om)} (0; v) \leq K_{r_j B_\epsilon} (0; v) = \frac{|v|}{r_j\epsilon} \rw 0$$
as $j\rw \infty$. This is a contradiction. Choosing a subsequence, we may assume that $\psi_j$ is convergent by Montel's theorem and $r_j \rw R_0$, $R_j\rw R_0$ for some $R_0 \geq 0$. Since $d \psi_j (x) = Id$, the limit $\psi$ of the sequence $\psi_j$ is also a 1-1 holomorphic mapping and $\psi (\Om) = R_0 I_\Om (x)$. Therefore, $R_0 >0$ and $\psi : \Om \rw R_0 I_\Om (x)$ is a biholomorphic mapping.
\end{proof}

Notice that the defintion of $c_\Om$ is similar to that of the squeezing function defined in \cite{DGZ1} (cf. \cite{DGZ2, JK, KZ}) which is motivated by \cite{LSY1, LSY2, Yeu}. The defintion of the squeezing function can be generalized to the case of more arbitrary circular models as follows (see also \cite{GP}). Let $D$ be a bounded complete circular domain in $\C^n$. Then the {\em squeezing function} $s^D_\Om (x)$ of $\Om$ at $x\in \Om$ modelled on $D$ is defined by 
\begin{eqnarray}\label{e;sqf}
s_\Om^D (x) &:=& \sup \left\{\frac{r}{R} : \exists \vp \in \mc B_\Om(x), R>0, r>0 \mbox{ such that }\right.\\
& &\hspace{2cm} \left.r D \subset \vp (\Om) \subset R D \right\}.\nonumber
\end{eqnarray}

Thanks to the similarity of definitions \eqref{e;mcf} and \eqref{e;sqf}, the following natural comparison between $c_\Om$ and $s^D_\Om$ can be verified. 

\begin{prop}\label{p;mcfsqf}
Let $\Om$ be a bounded domain in $\C^n$ and $x\in \Om$. Then for any bounded complete circular domain $D\subset \C^n$, 
$$c_\Om (x) \geq \left(s^D_\Om (x) \right)^2.$$
\end{prop}

\begin{proof}
Let $\vp\in \mc B_\Om (x)$ such that $\vp(\Om) \subset D.$ and $r>0$ be chosen that $rD \subset \vp(\Om).$ Since $D$ and $rD$ are completely circular, Barth's theorem (\cite{Bar}) implies that 
$$I_D(0) = D,\quad I_{rD}(0) = rD.$$
Then we have 
$$rD = I_{rD}(0) \subset I_{\vp(\Om)}(0) \subset I_D(0) = D$$
from the decreasing property of the Kobayashi metric. Therefore,
$$r I_{\vp(\Om)}(0)\subset rD\subset\vp(\Om) \subset D \subset \frac{1}{r}  I_{\vp(\Om)}(0).$$
This implies $c_\Om (x) \geq r^2.$ Taking the supremum for $r$, we get the conclusion.
\end{proof}

If $\Om$ is strongly pseudoconvex, then it turned out that the squeezing function $s_\Om(x) = s^{B^n}_\Om(x)$ modelled on the unit ball tends to $1$ as $x$ goes to a boundary point (\cite{DFW, KZ}). Therefore, we have immediately the following theorem as a corollary.

\begin{thm}\label{t;spc}
If $\Om$ is a bounded domain in $\C^n$ with $C^2$-smooth strongly pseudoconvex boundary, then $c_\Om (x)\rw 1$ as $x\rw p\in \di \Om$.
\end{thm}

Another example of domains is {\em polyhedral domains}. Recall that a bounded domain $\Om$ in $\C^n$ is called a polyhedral domain if there are smooth real-valued functions $\rho_1,\ldots, \rho_N$ defined on a neighborhood of $\ov\Om$ such that 
\begin{itemize}
\item $\Om = \{z\in\C^n : \rho_1(z) <0,\ldots, \rho_N(z)<0\}$,

\item on $\di_\al\Om := \{z\in\di\Om : \rho_\al(z)=0\}$ for $\al=1,\ldots,N$, $\nabla \rho_\al$ never vanishes,

\item each $\di_\al \Om $ is Levi flat.
\end{itemize}

$\Om$ is said to be {\em normal} if it satisfies in addition

\begin{itemize}
\item on each $\di_{\al_1\cdots \al_k}\Om :=\{z\in \di \Om: \rho_{\al_1}(z) = \cdots= \rho_{\al_k}(z) =0\}$, $\nabla \rho_{\al_1}$,..., $\nabla \rho_{\al_k}$ are linearly independent over $\C$.
\end{itemize}

Imitating the idea of \cite{Kim1, Kim2} about the scaling of convex normal polyhedral domains, we can prove the following.

\begin{thm}\label{t;poly}
Let $\Om = \{z\in\C^n : \rho_1(z) <0,\ldots, \rho_N(z)<0\}$ be a convex normal polyhedral domain in $\C^n$. Let
$q\in \di_{\al_1\cdots \al_n} \Om$ for some distinct indices $\al_1,\ldots,\al_n.$ Then $c_\Om (x) \rw 1$ as $x\rw q$. 
\end{thm}

\begin{proof}
Reordering indices, we may assume that $\{\al_1,\ldots,\al_n\} = \{1,\ldots,n\}$ without loss of generality. We can also assume that $|\nabla \rho_\al| \equiv 1$ on $\di_j\Om$ for each $\al=1,\ldots,N$. By Proposition \ref{p;mcfsqf}, we only need to prove that 
$$\sigma(x) := s^{\Delta^n}_\Om (x) \rw 1 \quad \mbox{as $x\rw q$},$$ where $\Delta^n$ represents the unit polydisc in $\C^n$. 
Notice that $\di_\al \Om$ is foliated by analytic varieties since it is Levi flat, and each of them is indeed a part of affine hyperplane because $\Om$ is convex. Therefore, there are $n$ hyperplanes $H_1,\ldots,H_n$ such that 
\begin{itemize}
\item $H_1\cap\cdots \cap H_n =\{q\}$,

\item $V_\al:=H_\al\cap \di\Om$ is the maximal subvariety in $\di_j\Om$ containing $q$.
\end{itemize}
Taking a linear affine change of coordinates, we assume that 
\begin{itemize}
\item $q= O = (0,\ldots,0)$ the origin,

\item $H_\al = \{z_\al =0\}$,

\item $\{\re z_\al =0\}$ is a supporting (real) hyperplanes of $\Om$ for $\al=1,\ldots,n$, 

\item $\Om \subset \{ \re z_1<0,\ldots \re z_n <0\} = \mc H^n$,
\end{itemize}
where $\mc H = \{\zeta \in \C : \re \zeta <0\}$, the left half plane in $\C$. Let us suppose that $x\in \Om$ is sufficiently close to $O$. Then among all affine hyperplanes which are tangent to $\di_\al\Om$, there is a unique $h_\al = h_\al(x)$ which is closest to $x$, $\al=1,\ldots,n$. Let $p = p(x) \in \di_{12\cdots n} \Om$ be the intersection point of those hyperplanes, that is, 
$$h_1 \cap\cdots\cap h_n = \{p\}.$$ Let 
$$v_\al =v_\al(x) =  \nabla \rho_\al (p),$$ and let $A_x$ be the affine transform defined by 
\begin{itemize}
\item $A_x(p) =0$,

\item $d A_x (v_\al) = e_\al$ for $\al=1,\ldots,n$.
\end{itemize}
It is obvious that $A_x \rw Id$ locally uniformly on $\C^n$ as $x\rw O$, since $p\rw O$ and $v_\al \rw e_\al$ as $x\rw O$ by the construction. Notice also that $A_x(\Om)\subset \mc H^n$ and $A_x(x) = (x_1,\ldots,x_n)$ where $x_\al<0$ and $x_\al\rw 0$ as $x\rw O$ for all $\al=1,\ldots,n$. Let $\Phi:\mc H^n \rw \Delta^n$ be the Cayley transform such that $\Phi(O) = (1,\ldots,1)$ and $\Phi(-1,\ldots,-1) = O = (0,\ldots, 0)$. Then $\Phi (A_x(x)) = (1-r_1,\ldots,1-r_n)$ for some $r_\al>0$ which tends to $0$ as $x\rw O$. The normality condition implies that $\di_{1\cdots n}\Om$ is a smooth real $n$-dimensional manifold and for each $\al=1,\ldots, n$, $\pi_\al (\Phi (A_x(\Om)))$ has a smooth boundary near $1$, where $\pi_\al$ denotes the natural projection $\C^n\ni (z_1,\ldots,z_n)\rw z_\al \in\C$. Therefore, there is $\epsilon>0$ such that $\Delta_\epsilon^n \subset \Phi (A_x(\Om))$ for all $x\in \Om$ which is sufficiently close to $O$, where $\Delta_\epsilon$ represents the disc in $\C$ centered at $1-\epsilon$ of radius $\epsilon$. Let $\Psi_x \in \Aut (\Delta^n)$ such that 
$$\Psi_x (1,\ldots, 1)=(1,\ldots,1) \quad\mbox{and}\quad \Psi_x (1-r_1,\ldots,1-r_n) = (0,\ldots,0)$$. 
For each $0<r<1$, then 
$$r\Delta^n \subset \Psi_x (\Delta_\epsilon^n) \subset \Psi_x \circ \Phi\circ A_x (\Om)$$ for any $x\in \Om$ which is suficiently close to $O$. Therefore,
$$\liminf_{\Om\ni x\rw O} s^{\Delta^n}_\Om (x) \geq r$$ for each $0<r<1$, and this yields the conclusion.
\end{proof}

\end{document}